\documentclass[twoside]{IEEEtran}
\pdfoutput=1

\usepackage{booktabs}
\usepackage{graphics}
\usepackage{amssymb,amsthm,amsmath}
\usepackage[pdftex,colorlinks,citecolor=black,linkcolor=black,urlcolor=black,bookmarks=false]{hyperref}
\def\arXiv#1{arXiv:\href{http://arXiv.org/abs/#1}{#1}}
\usepackage{doi}

\newtheorem{theorem}{Theorem}
\newtheorem{proposition}[theorem]{Proposition}
\newtheorem{lemma}[theorem]{Lemma}
\newtheorem{corollary}[theorem]{Corollary}

\theoremstyle{definition}
\newtheorem{definition}[theorem]{Definition}

\theoremstyle{remark}

\newcommand{\RR}{\mathbb{R}}
\newcommand{\FF}{\mathbb{F}}

\newcommand{\C}{\mathcal{C}}

\newcommand{\va}{\mathbf{a}}
\newcommand{\vb}{\mathbf{b}}
\newcommand{\vc}{\mathbf{c}}
\newcommand{\vf}{\mathbf{f}}

\newcommand{\vh}{\mathbf{h}}

\newcommand{\diff}{\Delta}  % or \triangle

\newcommand{\abs}[1]{\left\lvert #1 \right\rvert}

\newcommand{\paren}[1]{\left( #1 \right)}
\newcommand{\set}[1]{\left\{ #1 \right\}}

\begin{document}

\title{Energy-minimizing error-correcting codes}

\author{Henry Cohn and Yufei Zhao%
\thanks{Henry Cohn is with Microsoft Research New England, One Memorial Drive, Cambridge, MA 02142, USA (e-mail:
cohn@microsoft.com).}%
\thanks{Yufei Zhao is with the Department of Mathematics, Massachusetts Institute of
Technology, Cambridge, MA 02139, USA (e-mail: yufeiz@mit.edu).}%
\thanks{Zhao was supported by an internship at Microsoft Research
New England.}}

\maketitle

\begin{abstract}
We study a discrete model of repelling particles, and we show using linear
programming bounds that many familiar families of error-correcting codes
minimize a broad class of potential energies when compared with all other
codes of the same size and block length. Examples of these universally
optimal codes include Hamming, Golay, and Reed-Solomon codes, among many
others, and this helps explain their robustness as the channel model
varies.  Universal optimality of these codes is equivalent to minimality of
their binomial moments, which has been proved in many cases by Ashikhmin
and Barg.  We highlight connections with mathematical physics and the
analogy between these results and previous work by Cohn and Kumar in the
continuous setting, and we develop a framework for optimizing the linear
programming bounds. Furthermore, we show that if these bounds prove a code
is universally optimal, then the code remains universally optimal even if
one codeword is removed.
\end{abstract}

\begin{IEEEkeywords}
Error correction codes, Combinatorial mathematics.
\end{IEEEkeywords}

\section{Introduction} \label{sec:intro}

Analogies between discrete and continuous packing problems have long played a
key role in coding theory.  In this paper, we extend these analogies to
encompass discrete models of physics, by showing that certain classical codes
are ground states of natural physics models. In fact, they are ground states
of many different models simultaneously.  We call this phenomenon
\emph{universal optimality}, motivated by \cite{CK07}.

As we will explain after Lemma~\ref{lem:cm-basis}, a code is universally
optimal if and only if all the binomial moments of its distance distribution
are minimal.  This problem has been studied by Ashikhmin and Barg
\cite{AB99}, with a very different combinatorial motivation (namely, counting
pairs of codewords in subcodes with restricted support), and they gave some
important examples, such as Hamming, Golay, and Reed-Solomon codes. Thus,
universal optimality is not a new property. However, the physics motivation
appears to be new, and we provide new proof techniques. We also prove strong
structural results about these codes, including our most surprising theorem:
if the linear programming bounds prove a code is universally optimal, then it
remains universally optimal if any single codeword is removed.

Let $\FF_q$ denote an alphabet with $q$ elements, and let $|x-y|$ denote the
Hamming distance between words $x,y \in \FF_q^n$.  Of course, this notation
suggests that $\FF_q$ is a finite field, but we will make no use of the field
structure.

We view $\FF_q^n$ as a discrete model of the universe, and we envision a code
in $\FF_q^n$ as specifying the locations of some particles.  To separate
these particles from each other, we will let them repel each other.
Specifically, we will choose a pairwise potential function between the
particles, and then we will study the \emph{ground states} of this system,
i.e., the particle arrangements that minimize the total energy.

Given a code $\C \subseteq \FF_q^n$ and a function $f \colon \{1,2, \dots,
n\} \to \RR$, the \emph{potential energy} of $\C$ with respect to the
\emph{potential function} $f$ is defined to be
\[
E_f(\C) = \frac{1}{|\C|}\sum_{\substack{x, y \in \C \\ x \neq y}}
f\paren{\abs{x-y}}.
\]
The normalization factor of $1/|\C|$ is convenient but not essential.

Repulsive forces correspond to decreasing potential functions, and we wish
the repulsion to grow stronger as the points grow closer together. The
completely monotonic functions extend these properties in a particularly
compelling way. Let $\diff$ be the finite difference operator, defined by
$\diff f (n) = f(n+1) - f(n).$ A function $f \colon \{a, a+1, \dotsc, b\} \to
\RR$ is \emph{completely monotonic} if its iterated differences alternate in
sign via $(-1)^k \diff^k f \ge 0$; more precisely, $(-1)^k\diff^k f(i) \geq
0$ whenever $k\geq 0$ and $a \leq i \leq b-k$.

For example, the inverse power laws $f(r) = r^{-\alpha}$ with $\alpha>0$ are
completely monotonic.  To see why, note that their derivatives obviously
alternate in sign, and then the mean value theorem implies that the same is
true for finite differences.

\begin{definition}
A code $\C \subseteq \FF_q^n$ is \emph{universally optimal} if
\[
E_f(\C) \leq E_f(\C')
\]
for every $\C' \subseteq \FF_q^n$ with $\abs{\C'} = \abs{\C}$ and all
completely monotonic $f \colon \{1, \dotsc, n\} \to \RR$.
\end{definition}

Every universally optimal code $\C$ maximizes the minimal distance between
codewords given its size $\abs{\C}$, because it minimizes the energy under $r
\mapsto r^{-\alpha}$ as $\alpha \to \infty$.  However, universal optimality
is a far stronger condition than that.

The definition of universal optimality is analogous to that of Cohn and Kumar
\cite{CK07} in the continuous setting.  They studied particle arrangements in
spheres or projective spaces and showed that many beautiful configurations
are universally optimal, including the icosahedron, the $E_8$ root system,
and the minimal vectors in the Leech lattice.  More generally, universal
optimality helps explain the occurrence of certain remarkable symmetry groups
in discrete mathematics and physics \cite{C10}.

Bouman, Draisma, and van Leeuwaarden \cite{BDL13} have independently studied
energy minimization models on toric grids under the Lee metric. Their main
theorem implies universal optimality for certain checkerboard arrangements of
particles filling half of the grid, but they do not investigate other codes.

Universally optimal codes have robust energy minimization properties, which
translate into good performance according to a broad range of measures.  For
example, they minimize the probability of an undetected error under the
$q$-ary symmetric channel, provided that each symbol is more likely to remain
the same than to become any other fixed symbol.  (See Section~V of
\cite{AB99}.)

For another application, consider maximum-likelihood decoding for a
binary-input discrete memoryless channel.  The exact error probability for
decoding is subtle, but for relatively low-rate codes it is frequently
estimated using a union bound (see Theorem~7.5 in \cite[p.~153]{M04}). This
bound shows that the error probability for a random codeword from a code $\C$
is at most $E_f(\C)$ with $f(r) = \gamma^r$, where $\gamma$ is the
Bhattacharyya parameter.  Because $\gamma \le 1$, the potential function $f$
is completely monotonic, and thus a universally optimal code must minimize
this upper bound for the decoding error.  It does not necessarily minimize
the true decoding error \cite{HKL04}, but minimizing a useful upper bound is
nearly as good.

Optimality is by no means limited to this particular union bound. For
example, the same holds true for the AWGN channel with antipodal signaling.
(Verifying complete monotonicity for the potential function requires a brief
inductive proof, but it is not difficult.) This explains the observations of
Ferrari and Chugg \cite{FC03}, who used linear programming bounds to verify
that certain Hamming and Golay codes minimize this bound for a wide range of
signal-to-noise ratios. Our results prove that this always works and show how
to generalize it to other codes.

We will prove that all the codes listed in Table~\ref{tab:codes} are
universally optimal. (See the longer version \arXiv{1212.1913v1} of this
paper for a review of the definitions of these codes, as well as other
background and discussion removed for lack of space.) For the Hamming,
Hadamard, Golay, MDS, and Nordstrom-Robinson codes, universal optimality is a
theorem of Ashikhmin and Barg \cite{AB99}, as mentioned above.

\begin{table*}
\centering
\begin{minipage}[b]{0.75\textwidth}
\caption{LP universally optimal codes.} \label{tab:codes}

We write $a \to b$ to mean $a, a+1, a+2, \dots, b$ and $a \stackrel{2}{\to}
b$ to mean $a, a+2, a+4, \dots, b$. The justification numbers in square
brackets tell which lemmas or propositions imply LP universal optimality;
when there is no such number, the proof is by directly solving the linear
programs.

\bigskip

\centering
\begin{tabular}{lcccccc}

\toprule

Name [justification] & $q$ & $n$& $N$ & Support $\subseteq$ & Dual support $\subseteq$\\

\midrule

Binary Hamming [\ref{prop:quasicode-duality}, \ref{prop:UO-1d}] & $2$ &
$2^r-1$ & $2^{2^r-r-1}$ & $\scriptstyle
\set{3\to n-3, n}$ & $\scriptstyle \set{\frac{n+1}{2}}$\\

-- extended [\ref{prop:quasicode-duality}, \ref{prop:UO-design-cover}] & $2$
& $2^r$  & $2^{2^r-r-1}$ & $\scriptstyle \set{4
\stackrel{2}{\to} n-4}$ & $\scriptstyle \set{\frac{n}{2},n}$\\

-- even subcode [\ref{prop:quasicode-duality}, \ref{prop:UO-design-cover}] &
$2$ & $2^r -1$ & $2^{2^r-r-2}$ & $\scriptstyle \set{4 \stackrel{2}{\to} n-3}$
& $\scriptstyle \set{\frac{n-1}{2},
\frac{n+1}{2}, n}$\\

-- shortened [\ref{prop:quasicode-duality}, \ref{prop:UO-design-cover}] & $2$
& $2^{r} - 2$ & $2^{2^r-r-2}$ & $\scriptstyle \set{3
\to n-2}$ & $\scriptstyle \set{\frac{n}{2}, \frac{n}{2} + 1}$\\

-- $2\times$ shortened [\ref{prop:quasicode-duality}, \ref{prop:UO-3sup}] &
$2$ & $2^r-3$ & $2^{2^r-r-3}$ & $\scriptstyle \set{3 \to  n-1}$ &
$\scriptstyle\set{\frac{n-1}{2}, \frac{n+1}{2},
\frac{n+3}{2}}$\\

-- punctured [\ref{prop:quasicode-duality}, \ref{prop:UO-1d}] & $2$ & $2^{r}
- 2$ & $2^{2^r-r-1}$ & $\scriptstyle \set{2
\to n-2,n}$ & $\scriptstyle \set{\frac{n}{2}+1}$\\

$q$-ary Hamming [\ref{prop:quasicode-duality}, \ref{prop:UO-1d}] & $q$ &
$\frac{q^r - 1}{q-1}$ & $q^{n-r}$ &
$\scriptstyle \set{3 \to n}$ & $\scriptstyle \set{ q^{r-1}}$\\

-- shortened [\ref{prop:quasicode-duality}, \ref{prop:UO-1d}] & $q$ &
$\frac{q^r - q}{q-1}$ & $q^{n-r}$ & $\scriptstyle
\set{3 \to n}$ & $\scriptstyle \set{q^{r-1}-1, q^{r-1}}$\\

-- punctured [\ref{prop:quasicode-duality}, \ref{prop:UO-1d}] & $q$ &
$\frac{q^r - q}{q-1}$ & $q^{n-r+1}$ &
$\scriptstyle \set{2 \to  n}$ & $\scriptstyle \set{q^{r-1}}$\\

\midrule

Simplex (1-design) [\ref{prop:UO-1d}] & $q$& $n$  & $N$ &  $\scriptstyle
\set{ a}$ &
$\scriptstyle \set{2 \to n}$\\

-- punctured [\ref{prop:UO-1d}] & $q$& $n-1$  & $N$ &  $\scriptstyle \set{
a-1, a}$ &
$\scriptstyle \set{2 \to n-1}$\\

\midrule

Hadamard [\ref{prop:UO-design-cover}] & $2$ & $4k$  & $2n$ & $\scriptstyle
\set{\frac{n}{2}, n}$ &
$\scriptstyle \set{4 \stackrel{2}{\to} n-4}$\\

-- punctured [\ref{prop:UO-design-cover}] & $2$ & $4k - 1$ & $2n + 2$ &
$\scriptstyle \set{\frac{n-1}{2}, \frac{n+1}{2}, n}$ & $\scriptstyle \set{4
\stackrel{2}{\to} n-3}$\\

Conference [\ref{prop:UO-design-cover}] & $2$  & $4k + 1$ & $2n + 2$ &
$\scriptstyle \set{\frac{n-1}{2}, \frac{n+1}{2}, n}$ & $\scriptstyle \set{4
\stackrel{2}{\to} n-1}$\\

\midrule

Binary Golay [\ref{prop:quasicode-duality}, \ref{prop:UO-design-spread}]  &
$2$ & $23$  & $2^{12}$ & $\scriptstyle
\set{7,8,11,12,15,16,23}$ & $\scriptstyle \set{8,12,16}$\\

-- extended [\ref{prop:UO-design-spread}] & $2$ & $24$ & $2^{12}$ &
$\scriptstyle
 \set{8,12,16,24}$ &  $\scriptstyle \set{8,12,16,24}$\\

-- punctured [\ref{prop:quasicode-duality}, \ref{prop:UO-design-spread}] &
$2$ & $22$ &  $2^{12}$ & $\scriptstyle \set{6
 \to 16,22} \setminus \set{9,13}$ & $\scriptstyle
 \set{8,12,16}$\\

-- shortened & $2$ & $22$ & $2^{11}$ & $\scriptstyle
 \set{7,8,11,12,15,16}$ & $\scriptstyle
 \set{7,8,11,12,15,16}$\\

-- $2\times$ shortened & $2$ & $21$ & $2^{10}$ & $\scriptstyle
\set{7,8,11,12,15,16}$ & $\scriptstyle \set{6 \to 16} \setminus
\set{9,13}$\\

-- {punctured and} & $2$ & $20$ & $2^{10}$ & $\scriptstyle \set{6 \to 16}
\setminus \set{9,13}$ &  $\scriptstyle \set{6 \to 16} \setminus
\set{9,13}$\\
\phantom{--} $2\times$ shortened\\

Ternary Golay [\ref{prop:quasicode-duality}, \ref{prop:UO-design-spread}] &
$3$  & $11$ & $3^6$ & $\scriptstyle \set{5,6,8,9,11}$
& $\scriptstyle \set{6,9}$\\

-- extended [\ref{prop:UO-design-spread}] & $3$ & $12$  & $3^6$ &
$\scriptstyle \set{6,9,12}$ &
$\scriptstyle \set{6,9,12}$\\

-- shortened & $3$ & $10$ & $3^5$ & $\scriptstyle \set{5, 6,
8, 9}$ & $\scriptstyle \set{5, 6, 8, 9}$ \\

-- $2\times$ shortened [\ref{prop:UO-design-cover}] & $3$ & $9$ & $3^4$ &
$\scriptstyle \set{5, 6,
8, 9}$ & $\scriptstyle \set{4 \to  9}$\\

-- $3\times$ shortened [\ref{prop:UO-design-cover}] & $3$ & $8$ & $3^3$ &
$\scriptstyle \set{5, 6,
8}$ & $\scriptstyle \set{3 \to 8}$\\

-- $4\times$ shortened [\ref{prop:UO-1d}] & $3$ & $7$ & $3^2$ & $\scriptstyle
\set{5, 6}$
& $\scriptstyle \set{2 \to 7}$\\

-- punctured [\ref{prop:quasicode-duality}, \ref{prop:UO-design-spread}] &
$3$ & $10$ & $3^6$ & $\scriptstyle \set{ 4 \to 10}$ &
$\scriptstyle \set{6, 9}$ &\\

-- $2\times$ punctured [\ref{prop:quasicode-duality},
\ref{prop:UO-design-spread}] & $3$ & $9$ & $3^6$ & $\scriptstyle \set{3 \to
9}$ & $\scriptstyle \set{6, 9}$\\

-- $3\times$ punctured [\ref{prop:quasicode-duality}, \ref{prop:UO-1d}] & $3$
& $8$ & $3^6$ & $\scriptstyle \set{2 \to
8}$ & $\scriptstyle \set{6}$\\

\midrule

MDS [\ref{lem:pd-mds}, \ref{prop:UO-design-pd}]& $q$ &  $n$   & $q^{n-d+1}$ &
$\scriptstyle \set{d \to n}$ &
$\scriptstyle \set{n-d+2 \to  n}$\\

\midrule

Ovoid ($q>2$) [\ref{prop:UO-design-cover}] & $q$ & $q^2 + 1$ & $q^4$ &
$\scriptstyle \set{ q^2 - q,
q^2}$ & $\scriptstyle \set{4 \to n}$\\

-- shortened [\ref{prop:UO-design-cover}] & $q$ & $q^2$ & $q^3$ &
$\scriptstyle \set{q^2 - q, q^2}$
& $\scriptstyle \set{3 \to n}$\\

-- $2\times$ shortened [\ref{prop:UO-1d}] & $q$ & $q^2 - 1$ & $q^2$ &
$\scriptstyle
\set{q^2 - q}$ & $\scriptstyle \set{2 \to n}$\\

-- punctured [\ref{prop:UO-design-cover}] & $q$ & $q^2$ & $q^4$ &
\scalebox{0.8}{$\scriptstyle \scriptstyle \set{q^2-q-1, q^2-q, q^2 -1, q^2}$}
& $\scriptstyle \set{4
\to n}$\\

\midrule

Nordstrom-Robinson & $2$ & $16$  & $256$ & $\scriptstyle
\set{6,8,10,16}$ & $\scriptstyle \set{6,8,10,16}$\\

-- punctured & $2$  & $15$ & $256$ & $\scriptstyle \set{5 \to 10,15}$ &
$\scriptstyle \set{6,8,10}$\\

-- shortened  & $2$ &  $15$ & $128$ & $\scriptstyle
\set{6,8,10}$ & $\scriptstyle \set{5 \to 10,15}$ \\

-- $2\times$ shortened & $2$ & $14$ & $56$ & $\scriptstyle
\set{6,8,10}$ & $\scriptstyle \set{4 \to 10,14}$ \\

\bottomrule
\end{tabular}
\end{minipage}
\end{table*}

Universally optimal codes are common for short block lengths, but they become
increasingly rare for long block lengths. Brute force searches show that
there is a unique universally optimal binary code of size $N$ and block
length $n$ (up to translation and permutation of the coordinates) whenever $n
\le 4$ and $1 \le N \le 2^n$. For $n=5$, such a code exists if and only if $N
\not\in \{9, 12, 13, 14, 18, 19, 20, 23\}$, and it is unique except when
$N=5$ or $N=27$, in which case there are two isomorphism classes (see
Section~\ref{sec:further} for an explanation of the $N \leftrightarrow 32-N$
symmetry).  Thus, a universal optimum need not exist or be unique if it does
exist.

Our main technical tool for bounding energy is the linear program developed
by Delsarte \cite{Del72}, which was originally used to bound the size of
codes given their minimum distance and was applied to energy minimization and
related problems by Yudin \cite{Y92} and by Ashikhmin, Barg, and Litsyn
\cite{AB99,ABL01}. We will call a code \emph{LP universally optimal} if its
universal optimality follows from these bounds, as occurs for all the cases
in Table~\ref{tab:codes}.

Our most surprising theorem is that LP universally optimal codes continue to
minimize energy even after we remove a single codeword.  We know of no
continuous analogue of this property.  Furthermore, such codes are distance
regular (for each distance, every codeword has the same number of codewords
at that distance).

\begin{theorem} \label{thm:removept-intro}
Every LP universally optimal code is distance regular, and it remains
universally optimal when any single codeword is removed.
\end{theorem}

Removing a codeword yields a universal optimum, but the resulting code will
generally not be LP universally optimal.  Thus, this process cannot be
iterated.

\section{Linear programming bounds} \label{sec:LP}

We begin by formulating the linear programming bound for energy minimization.
Suppose $\C \subseteq \FF_q^n$. The Delsarte inequalities constrain the
\emph{distance distribution} $(A_0, A_1, \dots, A_n)$ of $\C$, where
\[
A_i = \frac{1}{\abs{\C}} \abs{\set{(x,y) \in \C^2 : \abs{x-y} = i}}
\quad \text{for $i = 0, 1, \dots, n$}.
\]
Specifically, let $K_k$ denote
the $k$-th Krawtchouk polynomial, defined by
\begin{equation}
\label{eq:krawdef}
\begin{split}
K_k(x) & = K_k(x;n,q)\\
& = \sum_{j=0}^k
(-1)^j(q-1)^{k-j}\binom{x}{j}\binom{n-x}{k-j}.
\end{split}
\end{equation}
Krawtchouk polynomials are orthogonal polynomials with respect to the
binomial distribution $\operatorname{Binom}(n;(q-1)/q)$. In other words,
\begin{equation}\label{eq:kraw-orthog}
\frac{1}{q^n} \sum_{i=0}^n \binom{n}{i} (q-1)^i K_j(i) K_k(i) = \binom{n}{j} (q-1)^j \delta_{jk}.
\end{equation}

The Delsarte inequalities are
\[
\sum_{i=0}^n A_i K_j(i) \ge 0
\]
for $j = 0, 1, \dots, n$ (see Theorem~3 in \cite{DL98}). Thus, the following
linear program in the variables $A_0,A_1,\dots,A_n$ gives a lower bound for
$E_f(\C)$ when $\abs{\C}=N$:
\begin{equation}\label{eq:LP}
  \begin{split}
    \text{minimize \ }\qquad   \sum_{i=1}^n A_if(i) &  \\
    \text{subject to}\qquad  \sum_{i=0}^n A_i K_j(i) &\geq 0 \qquad  \text{for $j = 1, 2, \dotsc, n$}, \\
    A_0 + A_1 + \dotsb + A_n &= N, \\
    A_0 &= 1, \\
    A_i &\geq 0 \qquad \text{for $i = 1, 2, \dots, n$}.
  \end{split}
\end{equation}

\begin{definition}
A code $\C \subseteq \FF_q^n$ is \emph{LP universally optimal} if its
distance distribution $(A_0, \dots, A_n)$ optimizes \eqref{eq:LP} for every
completely monotonic potential function $f$.
\end{definition}

Ashikhmin and Barg \cite{AB99} call LP universally optimal codes
\emph{extremal codes}.

For any fixed code, checking whether it is universally optimal or LP
universally optimal is a finite problem, since we can write down a basis for
the cone of completely monotonic functions as follows.

\begin{lemma} \label{lem:cm-basis}
The complete monotonic functions on $\set{0,1,\dots,n}$ are the nonnegative
span of the \emph{fundamental potential functions} $f_0, f_1, \dots, f_n$
defined by $f_j (x) = \binom{n-x}{j}$.
\end{lemma}

The potential energy with respect to $f_j$ is exactly the $j$-th
\emph{binomial moment} of the distance distribution, as defined by Ashikhmin
and Barg \cite{AB99}.  Thus, Lemma~\ref{lem:cm-basis} shows that a code is
universally optimal if and only if its binomial moments are minimal, so the
results of \cite{AB99} can be restated in terms of universal optimality.

Lemma~\ref{lem:cm-basis} includes $0$ in the domain of $f$, which will be
notationally convenient in Section~\ref{sec:quasicodes}, but we can always
extend $f$ from $\set{1,2,\dots,n}$ to $\set{0, 1, \dots, n}$ by setting
$f(0)$ to be a sufficiently large value that complete monotonicity continues
to hold.

We say that a function $f \colon \{a, a+1, \dotsc, b\} \to \RR$ is
\emph{absolutely monotonic} if all its finite differences are nonnegative;
i.e., $\diff^k f(i) \geq 0$ whenever $k \ge 0$ and $a \leq i \leq b-k$.

\begin{proof}[Proof of Lemma~\ref{lem:cm-basis}]
By changing $x$ to $n-x$, it suffices to prove that the functions $g_j(x) =
\binom{x}{j}$ span the cone of absolutely monotonic functions. Indeed,
$\diff^r g_j(x) = g_{j-r}(x)$ for $r \leq j$, and $\diff^r g_j(x) = 0$ for $r
> j$, so each $g_j$ is absolutely monotonic. Conversely, every function $g
\colon \set{0,1,\dots,n} \to \RR$ satisfies
\[
g(x) = \sum_{j=0}^n \binom{x}{j} \diff^j g(0) = \sum_{j=0}^n g_j(x) \diff^j g(0)
\]
by the discrete calculus analogue of the Taylor series expansion. If $g$ is
absolutely monotonic, then $\diff^j g(0) \geq 0$ for all $j$, as desired.
\end{proof}

Thus, checking whether a code of block length $n$ is LP universally optimal
amounts to solving $n$ linear programs (the $f_0$ case is trivial).  However,
checking whether a code is universally optimal seems far more difficult.

Linear programming duality transforms \eqref{eq:LP} into its dual as follows.
Here $c_0,\dots,c_n$ are the dual variables, and the equality conditions
follow from complementary slackness.

\begin{proposition} \label{prop:LP-bound}
Suppose $f \colon \{1, \dots, n\} \to \RR$ is any function, $h \colon \{0, 1,
\dots, n\} \to \RR$ satisfies
\[
h(i) \leq f(i) \qquad \text{for $i = 1, 2, \dots, n$},
\]
and there exist $c_0, c_1, \dots, c_n$ with $c_j \geq 0$ for $j \geq 1$ such
that
\[
h(i) = \sum_{j=0}^n c_j K_j(i) \qquad \text{for $i = 0, 1, \dots, n$}.
\]
Then every code $\C \subseteq \FF_q^n$ with $\abs{\C} = N$ has $f$-potential
energy at least $N c_0 - h(0)$.  Furthermore, equality holds if and only if
$h(i) = f(i)$ for all $i>0$ satisfying $A_i>0$ and $c_j=0$ for all $j>0$
satisfying $A_j^\perp > 0$, where $(A_i)$ is the distance distribution of
$\C$ and $(A^\perp_j)$ is the \emph{dual distance distribution} defined by
\begin{equation} \label{eq:dualdef}
A^\perp_j = \frac{1}{\abs{\C}} \sum_{i=0}^n A_i
K_j(i).
\end{equation}
\end{proposition}

\section{Quasicodes and duality}\label{sec:quasicodes}

In this section we show that LP universal optimality is preserved under the
duality operation expressed by \eqref{eq:dualdef}.  Curiously, this symmetry
seems to have no analogue in the continuous setting of \cite{CK07}.

We use the term \emph{quasicode} for a feasible point in the Delsarte linear
program, equipped with a duality operator called the MacWilliams transform
\cite[p.~137]{MS}.

\begin{definition}
A \emph{quasicode} $\va$ of length $n$ and size $N$ over $\FF_q$ is a real
column vector $(A_0, A_1, \dots, A_n)$ satisfying the constraints of the
linear program \eqref{eq:LP}.  In other words,
\[
\va \geq 0, \qquad K \va \geq 0, \qquad  \sum_{i=0}^n A_i = N, \qquad\text{and} \qquad A_0 = 1.
\]
Here $K$ stands for the matrix $(K_i(j))_{0 \leq i,j \leq n}$, and $\va \geq
0$ means that all coordinates of $\va$ are nonnegative. We write $\abs{\va}$
for the \emph{size} $N$ of the quasicode. Based on \eqref{eq:dualdef}, the
\emph{dual} of $\va$ is defined to be the quasicode
\[
\va^\perp = \frac{1}{\abs{\va}} K \va.
\]
\end{definition}

To see that $\va^\perp$ is a quasicode, we can use the identity $K^2 = q^n I$
(see (11), (12), and (17) in \cite{DL98}).  (The reason is that $K$ is the
radial Fourier transform and $K^2=q^nI$ is Fourier inversion.) It follows
that $\abs{\va^\perp} \abs{\va} = q^n$ and $\va^{\perp\perp} = \va$. For
every code $\C \subseteq \FF_q^n$, its distance distribution is a quasicode
$\va$ with $\abs{\va} = \abs{\C}$. Furthermore, if $\C$ is a linear code,
then its dual linear code $\C^\perp$ has distance distribution $\va^\perp$.

We say that $\va$ is a \emph{$t$-design} if its dual $\va^\perp$ satisfies
$A^\perp_j = 0$ for $1 \leq j \leq t$. Using Krawtchouk polynomials as a
basis for polynomials of degree at most $t$, one can check that $\va$ is an
$t$-design if and only if every polynomial $f$ of degree at most $t$
satisfies
\begin{equation} \label{eq:design-avg}
\frac{1}{N} \sum_{i=0}^n A_i f(i) = \frac{1}{q^n} \sum_{i=0}^n \binom{n}{i}(q-1)^i f(i).
\end{equation}

Given a potential function $f \colon \set{0,1,\dots,n} \to \RR$, let $\vf$ be
the column vector $(f(0), f(1), \dots, f(n))$.  Minimizing the $f$-potential
energy of a quasicode $\va$ amounts to minimizing the inner product
\[
\vf^t \va = \sum_{i=0}^n f(i)A_i.
\]
This quantity differs from the earlier definition of energy by including
$f(0)$, but it does not affect the notion of universal optimality since $A_0
= 1$, independently of $\va$.

\begin{definition}
A quasicode $\va$ of length $n$ over $\FF_q$ \emph{minimizes $f$-potential
energy} if $\vf^t \va \leq \vf^t \vb$ for every quasicode $\vb$ of length $n$
over $\FF_q$ with $\abs{\va} = \abs{\vb}$. It is a \emph{universally optimal
quasicode} if it minimizes $f$-energy for every completely monotonic $f$.
\end{definition}

Note that a code is LP universally optimal if and only if its distance
distribution is a universally optimal quasicode.  Universally optimal
quasicodes often exist in low dimensions; for example, they exist for all $n
\le 11$ and $1 \le N \le 2^n$. Nevertheless, they do not always exist.  For
example, there are no universally optimal quasicodes for $n=12$, $q=2$, and
$24 < N < 40$.

Given a quasicode $(A_0, \dots, A_n)$ with dual $(A^\perp_0, \dots,
A^\perp_n)$, we call $\set{i > 0 : A_i \neq 0}$ the \emph{support} of the
quasicode, and $\set{i > 0 : A^\perp_i \neq 0}$ the \emph{dual support} of
the quasicode. Of course we apply the same definitions to actual codes.

Proposition~\ref{prop:LP-bound} also applies to quasicodes, because its proof
used only the Delsarte inequalities. Note that the conditions for equality do
not take into account the actual values of the quasicode, but only its
support and dual support:

\begin{proposition} \label{prop:support-info}
Whether a quasicode is universally optimal depends only on its length, size,
support, and dual support.
\end{proposition}

A universally optimal quasicode is uniquely determined by its length and size
if it exists, because the energies with respect to the $n+1$ fundamental
potential functions put $n+1$ constraints on the quasicode, which are
linearly independent because there is one potential function of each degree.
Furthermore, if $\va$ is a universally optimal quasicode and $\vb$ is another
quasicode of the same length and size whose support and dual support are
respectively contained in those of $\va$, then $\vb$ is also universally
optimal, since the same $h$ that works for $\va$ also works for $\vb$.  Thus,
$\vb = \va$.

\begin{proposition} \label{prop:quasicode-duality}
Let $\va$ be a quasicode. Then $\va$ is universally optimal if and only if
its dual $\va^\perp$ is.
\end{proposition}

For example, the dual of an LP universally optimal linear code is LP
universally optimal.  We first show that complete monotonicity is preserved
under duality.

\begin{lemma} \label{lem:cm-dual}
If $\vf$ represents a completely monotonic function, then so does $K^t \vf$.
\end{lemma}

\begin{proof}
Recall from Lemma~\ref{lem:cm-basis} that the functions $f_j(x) =
\binom{n-x}{j}$ form a basis for the cone of completely monotonic functions.
Let $\vf_j$ denote the column vector corresponding to $f_j$.  To see that
$K^t$ leaves the cone of completely monotonic functions invariant, we will
use the identity
\begin{equation} \label{eq:cm-basis-dual}
K^t \vf_j = q^{n-j} \vf_{n-j}.
\end{equation}
Note that it can be rewritten as
\begin{equation}
\label{eq:binomial-duality}
\sum_{k=0}^n \binom{n-k}{j} K_k (i) = q^{n-j}\binom{n-i}{n-j}.
\end{equation}
We use the following generating function for Krawtchouk polynomials
\cite[p.~151]{MS}:
\[
\sum_{k=0}^n K_k(i) z^k = (1 + (q-1)z)^{n-i} (1-z)^{i}.
\]
By setting $z = (1 + w)^{-1}$ we can rewrite it as
\[
\sum_{k=0}^n K_k(i) (w+1)^{n-k} =  (w + q)^{n-i} w^i.
\]
Then \eqref{eq:binomial-duality} follows from comparing the coefficients of
$w^j$ in the above formula.
\end{proof}

\begin{proof}[Proof of Proposition~\ref{prop:quasicode-duality}]
Since the duality operator is an involution, it suffices to prove that if
$\va$ is universally optimal, then so is $\va^\perp$. Every quasicode can be
written as $\vb^\perp$ for some quasicode $\vb$. So it suffices to show that
$\vf^t \va^\perp \leq \vf^t \vb^\perp$ for every completely monotonic
potential $\vf$ whenever $\abs{\va} = \abs{\vb}$. By Lemma~\ref{lem:cm-dual},
$K^t \vf$ is also completely monotonic, and by the universal optimality of
$\va$ we have
\[
\begin{split}
\abs{\va} \vf^t \va^\perp
= \vf^t K \va
&= (K^t \vf)^t \va\\
&\leq (K^t \vf)^t \vb
= \vf^t K \vb
= \abs{\vb} \vf^t \vb^\perp.
\end{split}
\]
Therefore $\va^\perp$ is universally optimal.
\end{proof}

\section{Constructing dual solutions} \label{sec:construct}

To construct auxiliary functions $h$ for use in
Proposition~\ref{prop:LP-bound}, we will use polynomial interpolation. In
this section, we first review the theory of positive definite functions, and
then we prove inequalities on the values of interpolating polynomials.

\subsection{Positive definite functions} \label{sec:pdf}

For every function $h \colon \set{0, 1, \dotsc, n} \to \RR$, we can find
$c_0, c_1, \dots, c_n$ such that
\begin{equation} \label{eq:kraw-expansion}
h(i) = \sum_{j=0}^n c_j K_j(i) \qquad \text{for $i = 0, 1, \dots, n$}.
\end{equation}
Specifically, if $\vh$ is the column vector with entries $(h(i))_{0 \le i \le
n}$, then $\vh^t = \vc^t K$, so that $q^n \vc^t = \vh^t K$ as $K^2 = q^n I$.
Call $c_j$ the \emph{Krawtchouk coefficients} of $h$. For any $0 \leq s \leq
n$, the Krawtchouk polynomials $K_0, K_1, \dots, K_s$ span the polynomials of
degree at most $s$, so if $h$ is given by a polynomial of degree $s$, then
$c_j = 0$ for $j > s$.

Now we consider the requirement $c_j \geq 0$ from
Proposition~\ref{prop:LP-bound}.

\begin{definition}
A function $h \colon \set{0, 1, \dotsc, n} \to \RR$ is \emph{positive
definite} if its Krawtchouk coefficients are nonnegative.
\end{definition}

Such functions are called ``positive definite'' because they are the
functions for which $(h(\abs{x-y}))_{x,y \in \FF_q^n}$ is a positive
semidefinite matrix (see Theorem~2 in \cite{DL98}).

Proposition~\ref{prop:LP-bound} does not actually require $c_0 \geq 0$.
However, there seems to be little harm in assuming it. Doing so allows us to
use properties of positive definite functions such as the following standard
lemma, which follows from (24) in \cite{DL98}.

\begin{lemma} \label{lem:pd-prod}
The product of two positive definite functions is positive definite.
\end{lemma}

\begin{lemma} \label{lem:pd-big}
The function $h(x) = a - x$ is positive definite iff $a \geq (q-1)n/q$.
\end{lemma}

\begin{proof}
This assertion follows immediately from
\[
K_1(x) = (q-1)n - qx. \qedhere
\]
\end{proof}

\begin{corollary} \label{cor:pd-big-prod}
If $a_1, a_2, \dots, a_s \geq (q-1)n/q$, then $h(x) = (a_1 - x)(a_2 - x)
\dotsm (a_s - x)$ is positive definite.
\end{corollary}

\begin{lemma} \label{lem:pd-design}
Let $\va = (A_0, \dots, A_n)$ be a quasicode whose support consists of $a_1 <
a_2 < \cdots < a_s$ and suppose that $\va$ is a $(2s-1)$-design. Then $h(x) =
(a_1 - x)(a_2 - x) \dotsm (a_s - x)$ is positive definite.
\end{lemma}

\begin{proof}
Let $c_j$ be as in \eqref{eq:kraw-expansion}. Since $h$ is a degree $s$
polynomial, $c_j = 0$ for $j > s$.  To show that $c_s>0$, all we need to
check is that the leading coefficient of $K_s$ has sign $(-1)^s$, which is in
fact true for each term in \eqref{eq:krawdef}.  Now, for $j \leq s -1$, using
the fact that $\va$ is a $(2s-1)$-design and $h \cdot K_j$ is a polynomial of
degree at most $2s-1$, we have by the orthogonality \eqref{eq:kraw-orthog} of
the Krawtchouk polynomials and \eqref{eq:design-avg} that
\[
\begin{split}
(q-1)^j \binom{n}{j} c_j &= q^{-n} \sum_{i=0}^n \binom{n}{i} (q-1)^i h(i) K_j(i)\\
&= \frac{1}{N} \sum_{i=0}^n A_i h(i) K_j(i).
\end{split}
\]
The right side is nonnegative since $A_i h(i) = 0$ for $i \geq 1$ (because
$h$ vanishes on the support of $\va$) and $A_0 h(0)K_j(0) \geq 0$.
\end{proof}

\begin{lemma} \label{lem:pd-mds}
For $0 \leq j \leq n$, the function $h(x) = (n-j+1 - x)(n-j+2 - x) \cdots (n
- x)$ is positive definite.
\end{lemma}

\begin{proof}
We have $h(x) = j!f_j(x)$, where $f_j$ is the fundamental potential function
from Lemma~\ref{lem:cm-dual}. So $\vc^t = q^{-n} \vh^t K = q^{-n} j! \vf_j^t
K = q^{-j}j! \vf_{n-j}^t \geq 0$ by \eqref{eq:cm-basis-dual}.
\end{proof}

\subsection{Polynomial interpolation} \label{sec:interpolation}

We begin with an analogue of Rolle's theorem.

\begin{lemma} \label{lem:discrete-rolle}
Let $a < b$ be integers. If \[g \colon \set{a,a+1, \dots, b+1} \to \RR\]
satisfies $g(a)g(a+1) \leq 0$ and $g(b)g(b+1) \leq 0$, then $\diff g(c) \diff
g(c+1) \leq 0$ for some $a \leq c < b$.
\end{lemma}

\begin{proof}
Without loss of generality, we assume $g(a) \leq 0$ and $g(a+1) \geq 0$.
Since at least one of $g(b) \leq 0$ and $g(b+1) \leq 0$ is true, the sequence
$g(a+1), g(a+2), \dots, g(b+1)$ cannot be strictly increasing. If $c$ is the
smallest integer such that $g(c+1) \geq g(c+2)$, then $\diff g(c) > 0$ and
$\diff g(c+1) \leq 0$, as desired.
\end{proof}

\begin{lemma} \label{lem:discrete-rolle2}
Let $a_1 < a_2 < \cdots < a_r$ be integers. If a function
\[
g \colon
\set{a_1, a_1 + 1, a_1+2, \dots, a_r+1} \to \RR
\]
satisfies $g(a_i)g(a_{i+1}) \leq 0$ for $i = 1, 2, \dots, r$, then there is
some integer $c$ such that $a_1\leq c \leq a_r-r+1$ and $\diff^{r-1} g(c)
\diff^{r-1} g(c+1) \leq 0$.
\end{lemma}

\begin{proof}
This follows from repeatedly applying Lemma~\ref{lem:discrete-rolle}.
\end{proof}

\begin{lemma} \label{lem:interpolate2}
Let $f \colon \set{0,1,\dots,n} \to \RR$ be completely monotonic, let
$a_1,\dots,a_r \in \{0,1,\dots,n\}$ be distinct, and let $p$ be the unique
polynomial of degree less than $r$ such that $p(a_i) = f(a_i)$ for $i = 1, 2,
\dots, r$.  Then
\begin{equation} \label{eq:interp1}
(f(x) - p(x)) \prod_{i=1}^{r} (a_{i} - x) \geq 0
\end{equation}
for all $x = 0, 1, \dots, n$, and $p$ has the expansion
\begin{equation} \label{eq:interp2}
p(x) = \sum_{j=0}^{r-1} c_j \prod_{i=1}^{j} (a_{i} - x)
\end{equation}
with $c_0,\dots,c_{r-1} \ge 0$.
\end{lemma}

\begin{proof}
For \eqref{eq:interp1}, suppose $x \not\in \{a_1,\dots,a_r\}$, since
otherwise the inequality is trivial, and define $g\colon \set{0, 1, \dots, n}
\to \RR$ by
\begin{equation} \label{eq:gdef}
g(t) = f(t) - p(t) - A (t-a_1)(t-a_2) \cdots (t-a_r)
\end{equation}
with the constant $A$ chosen so that $g(x) = 0$; in other words,
\[
A = \frac{f(x) - p(x)}{\prod_{i=1}^r (x - a_i)}.
\]
We have $g(a_i) = 0$ for $i = 1, 2, \dots, r$ as well as $g(x) = 0$, so
Lemma~\ref{lem:discrete-rolle2} implies that there is some integer $c$ such
that $\diff^r g (c) \diff^r g(c+1) \leq 0$. Thus, $(-1)^r \diff^r g(c') \leq
0$ for either $c' = c$ or $c' = c+1$. Now, \eqref{eq:gdef} implies
\[
\diff^r g(c') = \diff^r f(c') - A r!
\]
and we have $(-1)^r \diff^r f(c') \geq 0$ by complete monotonicity, so
$(-1)^r A \geq 0$. Therefore,
\[
(f(x) - p(x)) \prod_{i=1}^r (a_i- x) = (-1)^r A \prod_{i=1}^r (x - a_i)^2
\geq 0.
\]

For \eqref{eq:interp2}, we solve for $c_0,\dots,c_{r-1}$ successively
starting with $c_0 = p(a_1) = f(a_1) \ge 0$.  Now for each $\ell$, the
polynomial $p_\ell$ defined by
\[
p_\ell(x) = \sum_{j=0}^{\ell-1} c_j \prod_{i=1}^j (a_i-x)
\]
is the unique polynomial of degree less than $\ell$ satisfying $p_\ell(a_i) =
f(a_i)$ for $i = 1, 2, \dots, \ell$. Applying \eqref{eq:interp1} to
$p_{\ell}$, we find that for $1 \le \ell \le r-1$,
\begin{align*}
0 &\leq (f(a_{\ell+1}) -  p_{\ell}(a_{\ell+1})) \prod_{i=1}^{\ell} (a_i-a_{\ell+1})\\
&=  (p_{\ell+1}(a_{\ell+1}) -  p_{\ell}(a_{\ell+1})) \prod_{i=1}^{\ell} (a_i-a_{\ell+1})\\
&= c_{\ell}\prod_{i=1}^{\ell} (a_i-a_{\ell+1})^2.
\end{align*}
It follows that $c_\ell \geq 0$, as desired.
\end{proof}

\section{Criteria for universal optimality} \label{sec:criteria}

We now use the inequalities from Section~\ref{sec:construct} to construct
auxiliary functions for use in Proposition~\ref{prop:LP-bound}.  These
results are applied in Table~\ref{tab:codes} as indicated by the lemma or
proposition numbers in square brackets in each line of the table.  For the
lines without references, we must resort to solving linear programs directly.

Recall that we do not include zero in the support of a quasicode.

\begin{definition}
Given a quasicode $\va$ of length $n$ over $\FF_q$, a \emph{pair covering} is
a subset $T \subseteq \set{1, 2, \dots, n}$ with elements $b_1 < b_2 < \dots
< b_t$ containing the support of $\va$ and such that $b_{2i-1} + 1 = b_{2i}$
whenever $2i \leq t$, while $b_t = n$ if $t$ is odd.
\end{definition}

\begin{proposition} \label{prop:UO-design-pd}
Let $\va$ be a quasicode of length $n$ and $T$ a pair covering of $\va$ with
elements $b_1 < b_2 < \cdots < b_t$. Then $\va$ is universally optimal if the
following two hypotheses are satisfied:
\begin{enumerate}
\item[(a)] The quasicode $\va$ is a $(t-1)$-design.
\item[(b)] For $1 \leq j \leq t-1$, the function $q_j(x) =
    \prod_{i=0}^{j-1} (b_{t-i}-x)$ is positive definite.
\end{enumerate}
\end{proposition}

We conjecture that condition (a) alone suffices.

\begin{proof}
Let $f \colon \set{0,1,\dots,n} \to \RR$ be completely monotonic, and let $h$
be the unique polynomial of degree less than $t$ such that $h(x) = f(x)$ for
all $x \in T$. We will show that $h$ satisfies the hypotheses of
Proposition~\ref{prop:LP-bound} and the conditions for equality.

For the inequality $f(x) \geq h(x)$, we apply \eqref{eq:interp1} with $a_i =
b_{t+1-i}$ and use the fact that $\prod_{i=1}^{t}(b_{i} - x) \geq 0$ for all
$x$ because $T$ is a pair covering.

To show that $h$ is positive definite, we write
\[
h(x) = \sum_{j=0}^{t-1} c_j \prod_{i=0}^{j-1} (b_{t-i} - x)
\]
with $c_j \ge 0$ by \eqref{eq:interp2} and $\prod_{i=0}^{j-1} (b_{t-i} - x)$
being positive definite by hypothesis (b).

All that remains is to check the complementary slackness conditions.  Because
$h(x) = f(x)$ for all $x \in T$, they are equal on the support of $\va$.
Because $\va$ is a $(t-1)$-design, the dual support is contained in
$\{t,\dots,n\}$ and hence the Krawtchouk coefficients of $h$ vanish on the
dual support.  Thus, $\va$ is universally optimal.
\end{proof}

Now we discuss two special cases in which part (b) of
Proposition~\ref{prop:UO-design-pd} is easy to verify using our results from
Section~\ref{sec:pdf}, as well as two elementary cases that do not fit into
the framework of Proposition~\ref{prop:UO-design-pd}.

\begin{proposition} \label{prop:UO-design-cover}
Let $\va$ be a quasicode of length $n$ over $\FF_q$, and let $T$ be a pair
covering of $\va$ with $\abs{T} = t$. If $\va$ is a $(t-1)$-design and at
most one element of $T$ is less than $(q-1)n/q$, then $\va$ is universally
optimal.
\end{proposition}

\begin{proof}
We only need to check condition (b) of Proposition~\ref{prop:UO-design-pd}.
Since at most one element of $T$ is less than $(q-1)n/q$, namely $b_1$, the
product $\prod_{i=0}^{j-1} (b_{t-i} - x)$ is positive definite by
Corollary~\ref{cor:pd-big-prod} for $1 \leq j \leq t-1$.
\end{proof}

\begin{proposition} \label{prop:UO-design-spread}
Let $\va$ be a quasicode of length $n$ over $\FF_q$. Suppose that $\va$ has
$s$ support elements and is a $(2s-1)$-design. Furthermore, suppose that
every two elements in the support differ by at least $2$, and at most one
element of the support is less than $(q-1)n/q$.  Then $\va$ is LP universally
optimal.
\end{proposition}

\begin{proof}
We shall construct a pair covering that satisfies the conditions of
Proposition~\ref{prop:UO-design-pd}. Suppose that nonzero elements of the
support are $a_1 < a_2 < \cdots < a_s$, so that $a_i \geq (q-1)n/q$ for all
$i \geq 2$. If $a_s < n$, then set
\[
\begin{split}
T &= \set{a_1 - 1, a_1} \cup \set{a_2, a_2 + 1} \cup \set{a_3, a_3 + 1}\\
&\qquad  \phantom{} \cup \cdots \cup \set{a_s, a_s + 1} % ad hoc spacing
\end{split}
\]
and if $a_s = n$, then set
\[
T = \set{a_1 - 1, a_1} \cup \set{a_2, a_2 + 1} \cup \set{a_3, a_3 + 1} \cup \cdots \cup \set{a_s}.
\]
By construction, $T$ is a pair covering. Let $t = \abs{T}$. When $a_s < n$,
we have $t = 2s$, and when $a_s = n$, we have $t = 2s - 1$. So $\va$ is
always a $(t-1)$-design and condition (a) of
Proposition~\ref{prop:UO-design-pd} is satisfied.

Now we check condition (b) of Proposition~\ref{prop:UO-design-pd}. In the
$a_s < n$ case, the partial product of an initial segment of \[(a_s + 1 -
x)(a_s - x) \cdots (a_2 + 1 - x)(a_2 - x)\] is positive definite by
Corollary~\ref{cor:pd-big-prod} since $a_j \geq (q-1)n/q$ for $j \geq 2$.
Furthermore
\[(a_s + 1 - x)(a_s - x) \cdots (a_2 + 1 - x)(a_2 - x)(a_1 - x)\] is positive
definite: \[(a_s - x)(a_{s-1}-x) \cdots (a_1 - x)\] is positive definite by
Lemma~\ref{lem:pd-design} and \[(a_s+1 - x)(a_{s-1}+1 -x) \cdots (a_2 + 1 -
x)\] is positive definite by Corollary~\ref{cor:pd-big-prod}, and so their
product is positive definite by Lemma~\ref{lem:pd-prod}. This completes the
$a_s < n$ case. The $a_s = n$ case is nearly identical. Thus, condition (b)
of Proposition~\ref{prop:UO-design-pd} is satisfied, and $\va$ is universally
optimal.
\end{proof}

\begin{proposition} \label{prop:UO-1d}
Let $\va$ be a quasicode. Suppose that $\va$ is a $1$-design, whose support
consists of a single integer or two consecutive integers. Then $\va$ is
universally optimal.
\end{proposition}

\begin{proof}[Sketch of proof]
We use a linear auxiliary function that agrees with the potential function on
the support, and on a neighboring point if the support has size one.
\end{proof}

\begin{proposition} \label{prop:UO-3sup}
Let $\va$ be a binary quasicode of length $n$.  Suppose $\va$ is supported at
$\set{0, a-1, a, a+1}$ where $a$ is odd, while its dual $\va^\perp$ satisfies
$A^\perp_1 = A^\perp_n = 0$. Then $\va$ is universally optimal.
\end{proposition}

\begin{proof}[Sketch of proof]
For a potential function $f$, one can check that the auxiliary function
\[
\begin{split}
&f(a-1) + \frac{1}{2} (f(a-1) - f(a+1)) (a - 1 - x)\\
&\quad \phantom{} + \frac{1}{4} ( f(a-1) - 2f(a) + f(a+1)) (K_n (x) - 1)
\end{split}
\]
works as $h(x)$ in Proposition~\ref{prop:LP-bound}.
\end{proof}

\section{Removing a codeword from a code} \label{sec:remove}

In this section, we show that removing a single codeword from an LP
universally optimal code always yields a universally optimal code. This
surprising fact will follow from a strengthening of the Delsarte linear
program due to Ashikhmin and Simonis \cite{AS98}.  It can fail without LP
universal optimality: in $\FF_2^2$, the three-point code
$\{(0,0),(0,1),(1,1)\}$ is universally optimal, but $\{(0,0),(0,1)\}$ is not.

\begin{proposition}[Ashikhmin and Simonis \cite{AS98}]
\label{prop:AS} Let $\C$ be a code of length $n$ over an alphabet of
size $q$ and such that $q$ does not divide $\abs{\C}$, and let
$(A_0,\dots,A_n)$ be its distance distribution. Then for $0 \leq j \leq
n$,
\[
\abs{\C} \sum_{i=0}^n A_i K_j(i)  \geq (q-1)^j \binom{n}{j}.
\]
\end{proposition}

See \arXiv{1212.1913v1} for a streamlined variant of the proof from
\cite{AS98}.

\begin{lemma} \label{lemma:qmult}
Let $f$ be any potential function.  If the Delsarte linear program proves
that a code $\C \subseteq \FF_q^n$ minimizes $f$-potential energy, then
either $|\C|$ is a multiple of $q$ or $f$ is minimized at all the distances
between pairs of distinct codewords in $\C$.
\end{lemma}

\begin{proof}
Suppose $|\C|$ is not a multiple of $q$.  Proposition~\ref{prop:AS} shows
that the dual distance distribution of $\C$ is strictly positive, and thus
the auxiliary function in Proposition~\ref{prop:LP-bound} must be constant.
Then the conclusion follows, because the auxiliary function is less than or
equal to $f$ everywhere and equal on the support of $\C$.
\end{proof}

\begin{corollary} \label{cor:multq}
Every LP universally optimal code in $\FF_q^n$ has size a multiple of
$q$ unless all pairs of distinct points in the code are at distance
$n$.
\end{corollary}

In the latter case, there can be at most $q$ points in the code.

\begin{proposition}
\label{prop:removept-f} Let $\C \subseteq \FF_q^n$ be a code and let $f
\colon \set{1, 2, \dots, n} \to \RR$ be any function (not necessarily
completely monotonic) such that the Delsarte linear programming bounds prove
$\C$ minimizes $f$-potential energy. Let $c \in \C$. Then
$E_f(\C\setminus\set{c}) \leq E_f(\C')$ for every code $\C'\subseteq \FF_q^n$
with $\abs{\C'} = \abs{\C}-1$.
\end{proposition}

\begin{proof}
By Lemma~\ref{lemma:qmult} we may assume that $|\C|$ is a multiple of $q$,
because the other case in the lemma is trivial. Let $N = |\C|$, let
$(1,A_1,\dots,A_n)$ be the distance distribution of $\C$, and let
$(1,B_1,\dots,B_n)$ be the expected distance distribution after removing a
random codeword from $\C$.  Given $(x,y) \in \C^2$ with $x \ne y$, the
probability that neither will be removed is $(N-2)/N$. Thus, $B_0=1$ while
for $i \ge 1$,
\[
\begin{split}
B_i &= \frac{1}{N-1} \cdot \frac{N-2}{N}
\,\abs{\set{(x,y) \in \C^2 : \abs{x-y} = i}}\\
& = \frac{N-2}{N-1} \, A_i. % ad hoc spacing
\end{split}
\]
Under this relationship between $A_i$ and $B_i$, the Delsarte inequalities
\[
\sum_{i=0}^n A_i K_j(i) \ge 0
\]
simply say
\[
K_j(0) + \frac{N-1}{N-2} \sum_{i=1}^n B_i K_j(i) \ge 0
\]
and hence are equivalent to
\[
(N-1)\sum_{i=0}^n B_i K_j(i) \ge K_j(0) = (q-1)^j \binom{n}{j}.
\]
The key observation underlying the proof is that these inequalities on
$(B_i)$ are exactly the Ashikhmin-Simonis inequalities from
Proposition~\ref{prop:AS}; i.e., $(A_i)$ satisfies the Delsarte inequalities
if and only if $(B_i)$ satisfies the Ashikhmin-Simonis inequalities.

Thus, our hypothesis that $(A_i)$ minimizes the $f$-potential energy
\[
\sum_{i=1}^n A_i f(i)
\]
among nonnegative vectors subject to the Delsarte inequalities, $A_0=1$, and
$\sum_i A_i = N$ implies that $(B_i)$ minimizes the expected energy
\[
\sum_{i=1}^n B_i f(i) = \frac{N-2}{N-1} \sum_{i=1}^n A_i f(i)
\]
subject to the Ashikhmin-Simonis inequalities, $B_0=1$, and $\sum_i B_i =
N-1$.

The Ashikhmin-Simonis inequalities apply to all codes of size $N-1$, because
$N$ is a multiple of $q$ and hence $N-1$ is not.  This means no code of size
$N-1$ in $\FF_q^n$ can have lower $f$-potential energy than the expected
energy after removing a random codeword from $\C$. Removing different
codewords might yield non-isomorphic codes, but by linearity of expectation
they must all have the same energy, since none of them can have lower energy
than the average. It follows that for every $c \in \C$, the code
$\C\setminus\set{c}$ minimizes $f$-potential energy among all codes of size
$\abs{\C}-1$.
\end{proof}

In particular, by letting $f$ vary over all completely monotonic functions,
we see that if $\C$ is LP universally optimal, then $\C\setminus\set{c}$ is
universally optimal for all $c \in \C$.  All of these codes
$\C\setminus\set{c}$ must have the same distance distribution, since they
have the same energy for all completely monotonic potential functions, which
span the space of all potential functions.  Thus $\C$ must be distance
regular.

This completes the proof of Theorem~\ref{thm:removept-intro}. We find the
result quite surprising, and the role of the Ashikhmin-Simonis inequalities
in the proof is mysterious.  It is natural to look for other proofs of these
inequalities.  There is a much simpler proof for binary codes (Theorem~5 in
\cite{BBMOS}), which we have been able to generalize to alphabets of prime
power order but no further. The elegant proof of the Delsarte inequalities in
\cite{SV91} can also be adapted to give a proof of the Ashikhmin-Simonis
inequalities, but in fact there is an error in \cite{SV91}: equation ($13''$)
is incorrect and the map $\sigma$ is not well defined for a general alphabet.
When the alphabet has prime power order, the proof works, but we see no way
to salvage it in general.

\section{Further questions and generalizations} \label{sec:further}

In the introduction we mentioned a $N \leftrightarrow 32-N$ symmetry for
codes of size $N$ in $\FF_2^5$.  The unoccupied locations in a code $\C
\subseteq \FF_q^n$ can be viewed as antiparticles, which are subject to
exactly the same forces as the original particles:
\[
\begin{split}
\paren{q^n-|\C|} E_f(\FF_q^n \setminus \C) & = |\C| \, E_f(\C)\\ % ad hoc spacing
& \quad \phantom{} + \paren{q^n-2|\C|} \sum_{k=1}^n \binom{n}{k} (q-1)^k f(k)
\end{split}
\]
by a simple inclusion-exclusion argument (see also Section~1.3.4 of
\cite{K07} for an essentially equivalent lemma).  Thus, $\FF_q^n \setminus
\C$ is universally optimal if and only if $\C$ is.

Linear programming bounds do not respect this antiparticle symmetry.  For
codes of size greater than $q^n/2$ in $\FF_q^n$, passing to the complement
can strengthen the Delsarte bounds, while for codes of size at most $q^n/2$
one can show that this yields no improvement.  Of course, few important codes
have size greater than $q^n/2$.

Beyond linear programming bounds and antiparticle symmetry, are there
systematic techniques that could be applied?  Semidefinite programming bounds
\cite{S05,GST06} are the most powerful approach known to proving coding
theory bounds. They have been applied to potential energy minimization in
projective space \cite{CW12}, but we have not investigated them in $\FF_q^n$.

Many of our results generalize straightforwardly to metric and cometric
association schemes, i.e., distance-regular graphs under the graph metric
with the ``$Q$-polynomial'' property \cite{DL98}.  There are several
noteworthy omissions, namely the theory of duality (including the definition
of the dual quasicode and Proposition~\ref{prop:quasicode-duality}) and the
results of Section~\ref{sec:remove}.  However, the results of
Sections~\ref{sec:construct} and~\ref{sec:criteria} all generalize if
$(q-1)n/q$ is replaced with the average distance between a pair of randomly
selected points in the graph, with the exception of Lemma~\ref{lem:pd-mds}
(which is needed only for MDS codes) and Proposition~\ref{prop:UO-3sup}.  The
proofs are essentially identical.

We have not attempted to compile an exhaustive list of examples for this more
general theory, along the lines of Table~\ref{tab:codes}, but there are
several interesting applications.  For example, consider the Johnson space of
binary vectors of length $n$ and weight $w$.  Every projective plane of order
$q$ yields an $S(2,q+1,q^2+q+1)$ Steiner system and thus a configuration of
$q^2+q+1$ points in the Johnson space with parameters $(n,w) =
(q^2+q+1,q+1)$.  This configuration is a simplex and a $2$-design, so it is
universally optimal.  The $S(5,8,24)$ Steiner system is a somewhat deeper
example.

The role of duality in association scheme theory is well understood (see
Section~2.6 in \cite{Del73}), and it does not generalize to arbitrary metric
and cometric association schemes.  However, the results of
Section~\ref{sec:remove} are far more mysterious, and we have no idea how far
they might generalize. In particular, we have no conceptual explanation for
why Proposition~\ref{prop:AS} turns out to be exactly what we require to
analyze removing one point from an LP universal optimum.  Any progress on
generalizing either the results or the proof techniques to other association
schemes would be exciting.

\section*{Acknowledgments}

We thank Alexei Ashikhmin, Alexander Barg, and the anonymous referee for
providing valuable feedback and suggestions.

\begin{IEEEbiographynophoto}{Henry Cohn}
is a principal researcher at Microsoft Research New England and adjunct
professor of mathematics at MIT.  His received his Ph.D.\ in mathematics from
Harvard in 2000, under the supervision of Noam Elkies.  His research is in
discrete mathematics, with connections to physics and computer science.
\end{IEEEbiographynophoto}

\begin{IEEEbiographynophoto}{Yufei Zhao}
is a Ph.D.\ student in the Department of Mathematics at the Massachusetts
Institute of Technology. He received his B.Sc.\ from MIT in 2010 and M.A.St.\
from the University of Cambridge in 2011. His research interests include
extremal and probabilistic combinatorics, as well as their applications to
other subjects such as number theory, probability, and geometry. He was a
recipient of the 2013 Microsoft Research PhD Fellowship.
% PhD has no periods in the fellowship name.
\end{IEEEbiographynophoto}

\vfill

\end{document}